\documentclass[12pt,letterpaper]{amsart}

\usepackage{amsmath,amsfonts,amssymb}
\usepackage[all]{xy}
\usepackage{tikz}
\usepackage{amsthm}
\usepackage[margin=2cm]{geometry}

\theoremstyle{plain} 
\newtheorem{theorem}{Theorem}[section]
\newtheorem{lemma}[theorem]{Lemma}

\theoremstyle{definition}
\newtheorem{definition}[theorem]{Definition}
\newtheorem{remark}[theorem]{Remark}
\newtheorem{example}[theorem]{Example}

\begin{document}

\title{A wide class of examples of pretorsion theories and related remarks}
\author{Zurab Janelidze}
\address{Department of Mathematical Sciences, Stellenbosch University, South Africa}
\address{
National Institute for Theoretical and Computational Sciences (NITheCS), Stellenbosch, South Africa }
\maketitle

\begin{abstract} 
In this paper we construct a wide class of examples of \emph{pretorsion theories} in the sense of A.~Facchini, C.~Finocchiaro, and M.~Gran. Given a category $\mathbb{C}$ with a terminal object $1$ and a category $\mathbb{D}$ with an initial object $0$, we show that $(\mathbb{C}\times \mathbf{0},\mathbf{1}\times\mathbb{D})$ is a pretorsion theory in $\mathbb{C}\times\mathbb{D}$ if and only if each morphism $0\to C$ in $\mathbb{C}$ is a monomorphism, and each morphism $D\to 1$ in $\mathbb{D}$ is an epimorphism. Here $\mathbf{0}$ denotes the set of initial objects in $\mathbb{D}$ and $\mathbf{1}$ denotes the set of terminal objects in $\mathbb{C}$. We then remark that the result generalised to products of arbitrary pretorsion theories.
\end{abstract}

\section{Introduction}

In a category $\mathcal{C}$, a \emph{pretorsion theory} in the sense of \cite{FFG21} is a pair $(\mathcal{T},\mathcal{F})$ of full replete subcategories of $\mathcal{C}$, such that for the ideal of morphisms that factor through objects in the intersection $\mathcal{Z}=\mathcal{T}\cap\mathcal{F}$, the following conditions hold:
\begin{itemize}
    \item[(T1)] Every morphism from an object in $\mathcal{T}$ to an object in $\mathcal{F}$ is a \emph{null morphism}, i.e., factors through an object in $\mathcal{Z}$.

    \item[(T2)] Every object $C$ of $\mathcal{C}$ is part of a short-exact sequence, i.e., a sequence $T\to C\to F$ of morphisms such that $T\to C$ is a kernel of $C\to F$ and conversely, $C\to F$ is a cokernel of $T\to C$. 
\end{itemize}
Here kernels and cokernels are relative to the ideal $\mathcal{N}$ of morphisms that factor through an object in $\mathcal{Z}$. These are by now well-established notions (see \cite{FFG21} and the references there), so we do not recall them here. 

The notion of a pretorsion theory is a vast generalisation of the classical notion of torsion theory in an abelian category introduced in \cite{Dic66}. There are various intermediate generalisations in the literature --- see \cite{FFG21,GJ20} and the references there. In \cite{FFG21}, it is shown that the notion of a pretorsion theory is highly versatile in the sense that, on one hand, many properties can be deduced from the simple definition, and on the other hand, there seems to be a big variety of examples.

In this paper we describe a wide class of examples of pretorsion theories which strongly makes use of the full generality of the notion of a pretorsion theory defined in \cite{FFG21}. The nature of these examples is then clarified using products of pretorsion theories. We conclude with a couple of remarks motivating the study of the category of pretorsion theories (with varying underlying category). 

\section{The special case}

In the notation used in the Abstract, let us explore what it takes for $(\mathbb{C}\times \mathbf{0},\mathbf{1}\times\mathbb{D})$ to be a pretorsion theory in the sense of \cite{FFG21}. Note that:
\begin{itemize}
    \item $\mathbb{C}\times \mathbf{0}$ and $\mathbf{1}\times\mathbb{D}$ are each closed under isomorphisms in $\mathbb{C}\times\mathbb{D}$.

    \item The intersection $\mathcal{Z}=(\mathbb{C}\times \mathbf{0})\cap(\mathbf{1}\times\mathbb{D})=\mathbf{1}\times\mathbf{0}$ consists of all $(T,I)$, where $T$ is a terminal object in $\mathbb{C}$ and $I$ is an initial object in $\mathbb{D}$.

    \item Therefore, a morphism $(A,I)\to (T,B)$ with $(A,I)\in\mathbb{C}\times\mathbf{0}$ and $(T,B)\in\mathbf{1}\times\mathbb{D}$ is always null (i.e., it factorizes through an object in $\mathcal{Z}$, namely, through $(T,I)$).
\end{itemize}
So, for $(\mathbb{C}\times \mathbf{0},\mathbf{1}\times\mathbb{D})$ to be a pretorsion theory, it is necessary and sufficient that every object $(C,D)$ is part of a short exact sequence:
\begin{equation}\label{EquA}\xymatrix@=30pt{(X,I)\ar[r]^-{(m_1,m_2)} & (C,D)\ar[r]^-{(e_1,e_2)} & (T,Y)}\end{equation}
where $I\in\mathbf{0}$ and $T\in\mathbf{1}$.
We note that (\ref{EquA}) is a short exact sequence if and only if the following conditions hold:
\begin{itemize}
\item[(E1)] $(m_1,m_2)$ is a monomorphism. 

\item[(E2)] $(e_1,e_2)$ is an epimorphism. 

\item[(E3)] The composite $(e_1,e_2)\circ (m_1,m_2)$ is null; indeed it is, as it factors through $(I,T)\in\mathcal{Z}$.

\item[(E4)] Whenever a composite $(e_1,e_2)\circ(u_1,u_2)$ is null, we have $(u_1,u_2)=(m_1,m_2)\circ (u'_1,u'_2)$ for some $u'_1,u'_2$.

\item[(E5)] Whenever a composite $(v_1,v_2)\circ(m_1,m_2)$ is null, we have $(v_1,v_2)=(v'_1,v'_2)\circ (e_1,e_2)$ for some $v'_1,v'_2$.
\end{itemize}
We prove the following:

\begin{lemma}\label{LemA}
The diagram (\ref{EquA}) in $\mathbb{C}\times\mathbb{D}$, where $I\in\mathbf{0}$ and $T\in\mathbf{1}$, is a short exact sequence if and only if $m_1$ is an isomorphism, $m_2$ is a monomorphism, $e_1$ is an epimorphism and $e_2$ is an isomorphism.     
\end{lemma}

\begin{proof} First, we make some general observations:
\begin{itemize}
    \item[(G)] $(f_1,f_2)$ is a monomorphism/epimorphism if and only if so are each $f_i$.
\end{itemize}
Now, suppose (\ref{EquA}) is a short exact sequence. (G) together with (E1-2) leave us to show that $m_1$ is a split epimorphism and symmetrically, $e_2$ is a split monomorphism. Thanks to commutativity of the right-hand side square below and (E4), we get a commutative triangle on the left:
$$
\xymatrix@=30pt{(X,I)\ar[r]^-{(m_1,m_2)} & (C,D)\ar[r]^-{(e_1,e_2)} & (T,Y)\\ & (C,I)\ar[ul]^-{(u'_1,1_I)}\ar[u]_-{(1_C,m_2)}\ar[r]_-{(e_1,1_I)} & (T,I)\ar[u]_-{(1_T,e_2m_2)}} $$
Then $m_1u'_1=1_C$, proving that $m_1$ is a split epimorphism. That $e_2$ is a split monomorphism can be proved similarly (with a dual-symmetric argument).

For the converse, suppose all the assumptions stated for $m_i,e_i$ in the lemma hold. It is not difficult to see that since $m_1$ and $e_2$ are isomorphisms, without loss of generality we can assume that they are actually identity morphisms. So the sequence (\ref{EquA}) becomes the top line of the following diagram:
$$
\xymatrix@=30pt{(C,I)\ar[r]^-{(1_C,m_2)} & (C,D)\ar[r]^-{(e_1,1_D)} & (T,D)\\ & (U_1,U_2)\ar[u]_-{(u_1,u_2)}\ar[r]_-{(e_1u_1,d)}\ar@{-->}[ul]^-{(u_1,d)} & (T,I)\ar[u]_-{(1_T,m_2)}} $$
(E1-2) hold by (G) and we already know (E3) holds. To prove (E4), suppose $(e_1,1_D)\circ (u_1,u_2)$ factors through an object in $\mathcal{Z}$. Without loss of generality, we can assume that this object is $(T,I)$. This gives the rest of the above diagram of solid arrows. Now, $(u_1,u_2)$ indeed factors through $(1_C,m_2)$ as required in (E3): look at the dashed arrow in the diagram above. The proof of (E5) is similar.
\end{proof}

We are ready to prove the result announced in the abstract:

\begin{theorem}\label{TheA}
$(\mathbb{C}\times \mathbf{0},\mathbf{1}\times\mathbb{D})$ is a pretorsion theory if and only if every morphism $C\to 1$ in $\mathbb{C}$ is an epimorphism and every morphism $0\to D$ in $\mathbb{D}$ is a monomorphism.
\end{theorem}

\begin{proof} Suppose $(\mathbb{C}\times \mathbf{0},\mathbf{1}\times\mathbb{D})$ is a pretorsion theory. Then any object $(C,D)$ is part of a short exact sequence, which by Lemma~\ref{LemA} can be written out as
$$\xymatrix@=30pt{(C,0)\ar[r]^-{(1_C,m_2)} & (C,D)\ar[r]^-{(e_1,1_D)} & (1,D)}$$
where $m_2\colon C\to D$ is a monomorphism and $e_1\colon C\to 1$ is an epimorphism. To prove the converse, we can apply Lemma~\ref{LemA} again and use the above sequence for each $(C,D)$. Thanks to the discussion at the start of the section, this would conclude the proof.
\end{proof}

So in such pretorsion theory, the ``torsion part'' of an object $(C,D)$ is given by $(C,0)$, whereas the ``torsion-free part'' is given by $(1,D)$. 

\begin{remark}
The property that every morphism going out from an initial object $0$ is a monomorphism, under the presence of a terminal object $1$, is equivalent to the property that the unique morphism $0\to 1$ is a monomorphism. This is a useful generalisation of pointedness in categorical algebra, first emphasized in \cite{Bou01} (see also \cite{GJ17}).    
\end{remark}

\begin{example}\label{ExaA}
Let $\mathbb{C}=\mathbf{Set}^\mathsf{op}$ and $\mathbb{D}=\mathbf{Set}$. We can think of objects $(X,Y)$ in $\mathbf{Set}^\mathsf{op}\times \mathbf{Set}$ as sets having two types of elements: \emph{positive} elements (elements of $Y$) and \emph{negative} elements (elements of $X$). To emphasize this intuition, we can write $Y-X$ for $(X,Y)$. Theorem~\ref{TheA} is applicable here, since $\varnothing\to S$ is an injective function, for any set $S$. The torsion objects in the corresponding pretorsion theory are \emph{negative sets}, i.e., sets having only negative elements, while torsion-free objects are \emph{positive sets}, i.e., sets having only positive elements. The short exact sequence that every object $Y-X$ fits into identified the ``subset'' $\varnothing-X$ of $Y-X$ consisting of negative elements, and the subset $Y-\varnothing$ of $X-Y$ consisting of positive elements:
$$\xymatrix{\varnothing-X\ar[r] & Y-X\ar[r] & Y-\varnothing}$$
\end{example}

\section{Thin pretorsion theories}

As shown in \cite{FFG21}, the torsion and the torsion-free part of an object are, in general, unique (up to canonical isomorphisms). In our case, the reverse is also two: two objects having isomorphic torsion and torsion-free pair of objects will necessarily be isomorphic. Indeed, a much stronger property holds.

\begin{definition} Call a pretorsion theory $(\mathcal{T},\mathcal{F})$ in a category $\mathcal{C}$ a \emph{thin} pretorsion theory, when the functor $\mathcal{C}\to \mathcal{T}\times \mathcal{F}$ which assigns to each object $X$ the pair $(t(X),f(X))$ of its torsion and torsion-free part\footnote{See \cite{FFT21} for the constructions of the functors $f$ and $t$.}, is an equivalence of categories.
\end{definition}

That our pretorsion theories are thin is obvious after observing that the functor $\mathcal{C}\to \mathcal{T}\times\mathcal{F}$ from the definition above is the equivalence $$T\times F\colon\mathbb{C}\times\mathbb{D}\to (\mathbb{C}\times \mathbf{0})\times (\mathbf{1}\times\mathbb{D}),$$ where $T$ and $F$ are equivalences given by $T(C)=(C,0)$ and $F(D)=(1,D)$.

\section{The general case}

Our construction of a pretorsion theory can be generalised significantly, as witnessed by the following results.

\begin{lemma}
For any category $\mathbb{C}$ having a terminal object, the pair $(\mathbb{C},\mathbf{1})$ where $\mathbf{1}$ is the class of terminal objects in $\mathbb{C}$, is a pretorsion theory if and only if every morphism to the terminal object is an epimorphism. Such torsion theories are thin.
\end{lemma}

\begin{proof}
This lemma has a simple direct proof. However, we can also derive it from Theorem~\ref{TheA}. Consider the trivial category $\mathbf{0}=\{\ast\}$ and apply Theorem~\ref{TheA} to $\mathbb{C}$ and $\mathbb{D}=\mathbf{0}$. Then we obtain that every morphism to the terminal object in $\mathbb{C}$ is an epimorphism if and only if $(\mathbb{C}\times \{\ast\},\mathbf{1}\times\{\ast\})$ is a pretorsion theory in $\mathbb{C}\times\{\ast\}$. The statement of the lemma then follows thanks to the canonical isomorphism $\mathbb{C}\approx\mathbb{C}\times\{\ast\}$.  
\end{proof}

Dually and symmetrically, we have:

\begin{lemma}
For any category $\mathbb{D}$ having an initial object, the pair $(\mathbf{0},\mathbb{D})$, where $\mathbf{0}$ is the class of initial objects in $\mathbb{D}$, is a pretorsion theory if and only if every morphism from the initial object is a monomorphism. Such torsion theories are thin.
\end{lemma}

Theorem~\ref{TheA} is then a consequence of the two lemmas above and the following result, the proof of which we omit since it is a simple routine, especially after seing the proof of Theorem~\ref{TheA}.

\begin{theorem}\label{TheB}
Consider a non-empty family $\mathbb{C}=(\mathbb{C}_i)_{i\in I}$ of non-empty categories and a family $(\mathcal{T}_i,\mathcal{F}_i)_{i\in I}$ of pairs of categories. The pair
$$(\Pi\mathcal{T},\Pi\mathcal{F})=\left(\prod_{i\in I}\mathcal{T}_i,\prod_{i\in I}\mathcal{F}_i\right)$$
is a pretorsion theory in $\Pi\mathbb{C}=\prod_{i\in I}\mathbb{C}_i$ if and only if each $(\mathcal{T}_i,\mathcal{F}_i)$ is a pretorsion theory in $\mathbb{C}_i$. Moreover, each of these pretorsion theories is thin if and only if the former one is.  
\end{theorem}

\begin{proof} Firstly, we remark that when either of the two conditions hold, each $\mathcal{T}_i,\mathcal{F}_i$ are non-empty. So without loss of generality we may assume that these categories are non-empty from the onset. 
Under this assumption, each $\mathcal{T}_i,\mathcal{F}_i$ are full replete subcategories of $\mathbb{C}_i$ if and only if $\Pi\mathcal{T},\Pi\mathcal{F}$ are full replete subcategories of $\Pi\mathbb{C}$.

We have:
$$\Pi\mathcal{T}\cap \Pi\mathcal{T}=\Pi_{i\in I}(\mathcal{T}_i\cap \mathcal{F}_i).$$
This implies that a morphism is null in $\Pi\mathbb{C}$ if and only if it is null in each $\mathbb{C}_i$. So null-morphisms in $\Pi\mathbb{C}$ are \emph{component-wise}. This clearly guarantees that (T1) holds for $(\Pi\mathcal{T},\Pi\mathcal{F})$ if and only if it holds for each $(\mathcal{T}_i,\mathcal{F}_i)$. To guarantee the same for (T2) it suffices to show that short exact sequences in $\Pi\mathbb{C}$ are also component-wise (an analogue, and in fact a generalisation of Lemma~\ref{LemA}). It is easy to see that if each 
$$\xymatrix{T_i\ar[r]^-{m_i} & C_i\ar[r]^-{e_i} & F_i}$$
is a short exact sequence, then so is 
$$\xymatrix{(T_i)_{i\in I}\ar[r]^-{(m_i)_{i\in I}} & (C_i)_{i\in I}\ar[r]^-{(e_i)_{i\in I}} & (F_i)_{i\in I}.}$$
Conversely, suppose the above is a short exact sequence. To show that the previous sequence is also short exact for some $i\in I$, consider a morphism $u_i\colon U_i\to C$ such that $e_iu_i$ is null. Define $U_j=T_j$ when $j\neq i$ and $u_j=m_j$ when $j\neq i$. Then the composite $(e_i)_{i\in I}(u_i)_{i\in I}$ is null. This results in a factorisation of $(u_i)_{i\in I}$ through $(m_i)_{i\in I}$, and hence a factorisation of $u_i$ through $m_i$. By the fact that being a monomorphism is a component-wise property and duality, we obtain the desired: that the sequence at $i$ shown above is short exact.

Next, we prove that thinness is preserved and reflected under products. It is easy to see that the functor $K\colon \Pi\mathbb{C}\to \Pi\mathcal{T}\times \Pi\mathcal{F}$ that assigns to each objects its torsion and torsion-free part can be built component-wise using similar functors at each $K_i$. In fact, we have a commutative diagram
$$\xymatrix{\Pi\mathbb{C}\ar[r]^-{K}\ar[rd]_-{\Pi_{i\in I}(K_i)\;\;\;} & \Pi\mathcal{T}\times \Pi\mathcal{F}\ar[d]^-{L}\\ & \Pi_{i\in I}(\mathcal{T}_i\times \mathcal{F}_i)}$$
where $L$ is an isomorphism. So $K$ will be an equivalence if and only if each $K_i$ is an equivalence. 
\end{proof}

Between Theorem~\ref{TheA} and the theorem above, there is also the following interesting intermediate result.

\begin{theorem} For a category $\mathbb{C}$, its full replete subcategory $\mathbf{1}$, and a category $\mathbb{D}$ with a full replete subcategory $\mathbf{0}$, the pair 
$(\mathbb{C}\times \mathbf{0},\mathbf{1}\times\mathbb{D})$ is a pretorsion theory in $\mathbb{C}\times\mathbb{D}$ if and only if $\mathbf{1}$ is an epireflective subcategory of $\mathbb{C}$ and $\mathbf{0}$ is a monocoreflective subcategory of $\mathbb{D}$. Such torsion theories are thin.
\end{theorem}

This theorem is a consequence of the previous one and the following lemma along with its symmetric dual, which generalize the two lemmas given above.

\begin{lemma}\label{LemB}
For any category $\mathbb{C}$, a pair $(\mathbb{C},\mathbf{1})$ is a pretorsion theory in $\mathbb{C}$ if and only if $\mathbf{1}$ is a full replete epireflective subcategory of $\mathbb{C}$. Every such torion theory is thin.
\end{lemma}

\begin{example}\label{ExaB} Consider the category $\mathbf{Set}_\mathrm{Rel}$ of sets and relations between sets as morphisms. Then $(\mathbf{Set}_\mathrm{Rel},\mathbf{Set}_\mathrm{Rel})$ is a pretorsion theory (it can be obtained by Lemma~\ref{LemB}). 
Apply Theorem~\ref{TheB} to get a pretorsion theory which is the product of this pretorsion theory with the one from Example~\ref{ExaA}. Objects of the resulting category can be thought of as sets having three types of elements: negative, positive (as in Example~\ref{ExaA}), and neutral (those contributed by $\mathbf{Set}_\mathrm{Rel}$). Torsion part of an object can be obtained by discarding the positive elements, while torsion-free part by discarding the negative elements. 
\end{example}

\section{Concluding remarks}

The product construction exhibited in Theorem~\ref{TheB} can easily be realised as a product in a suitable category of pretorsion theories. Define a morphism of pretorsion theories $(\mathcal{C},\mathcal{T},\mathcal{F})\to (\mathcal{C}',\mathcal{T}',\mathcal{F}')$ to be a functor $M\colon \mathcal{C}\to\mathcal{C}'$ that preserves torsion and torsion-free objects, as well those short exact sequences where the left object is torsion and the right object is torsion-free. We leave exploration of this category, as well the relation to its subcategory of thin pretorsion theories for future work. We only make two remarks here that should be useful for such exploration:
\begin{itemize}
\item It seems that the category of pretorsion theories (or more widely, of multi-pointed torsion theories), itself has a useful pretorsion theory, where the torsion pretorsion theories are of the form $(\mathcal{C},\mathcal{C},\mathcal{F})$, i.e., precisely those captured by Lemma~\ref{LemB}, while torsion-free pretorsion theories are their symmetric duals, i.e., of the form $(\mathcal{C},\mathcal{T},\mathcal{C})$.

\item One must perhaps consider a wider category of \emph{multi-pointed torsion theories}, defined in a similar way as pretorsion theories, but using instead 
a quadruple $(\mathcal{C},\mathcal{T},\mathcal{F},\mathcal{N})$ where $\mathcal{N}$ is any ideal and is used in the place of the ideal of morphisms factoring through objects in $\mathcal{T}\cap \mathcal{F}$.
\end{itemize}
The term ``multi-pointed'' comes from \cite{GJU12}, where it is used to refer to a category equipped with an ideal of morphisms. Coincidentally, it is used in \cite{GJ20} as well, where a notion of torsion theory suggested above is introduced, however, with a few restrictions:
\begin{itemize}
    \item[(T0)] $\mathcal{N}$ is a required to be a ``closed ideal'', which means that a morphism belongs to $\mathcal{N}$ if and only if it factors through an object whose identity morphism belongs to $\mathcal{N}$.

    \item[(T3)] Every identity morphism is required to have a kernel and a cokernel.
\end{itemize}
It is then shown in \cite{GJ20} (see Corollary~2.3 there) that the objects whose identity is a null morphism are precisely those that are both torsion and torsion-free. So pretorsion theories in the sense of \cite{FFG21} can be seen as torsion theories in multi-pointed categories in the sense of \cite{GJ20} where (T3) is dropped. We suggest to drop (T0) as well.

The examples presented in this paper take advantage of dropping (T3). For instance, neither Example~\ref{ExaA} nor Example~\ref{ExaB} satisfy (T3) for the simple reason that in the category of sets, non-empty sets do not admit a morphism to the empty set (see also Remark~7.3 in \cite{FFG21}). The property (T3) behaves with products similarly to the other properties, so if a pretorsion theory does not satisfy (T3) neither will its product with another pretorsion theory (irrespective of whether the latter has the property (T3) or not).

\end{document}